\documentclass[12pt]{article}

\usepackage{amsmath}
\usepackage{amsfonts,amssymb,amsthm,overpic}
\makeatletter
\newcommand{\diam}{\mathop{\operator@font diam}}
\usepackage{setspace}
\usepackage{multicol}
\usepackage{multirow}
\usepackage[compact]{titlesec}

\begin{document}

\def\s{\subseteq}
\def\h{\widehat}
\def\v{\varphi}
\def\t{\widetilde}
\def\ov{\overline}
\def\L{\Lambda}
\def\l{\lambda}
\def\O{\Omega}
\def\H{I\!\! H}
\def\a{\approx}
\def\k{\widetilde}
\def\la{\lambda}
\def\d{\delta}
\def\L{\Lambda}
\def\O{\Omega}
\def\r{\rho}
\def\ov{\overline}
\def\un{\underline}
\newcommand{\cT}{\mathcal{T}}
\newcommand{\cR}{\mathcal{R}}
\newcommand{\cL}{\mathcal{L}}
\newcommand{\cU}{\mathcal{U}}
\newcommand{\cV}{\mathcal{V}}
\newcommand{\cW}{\mathcal{W}}
\newcommand{\cM}{\mathcal{M}}
\newcommand{\cN}{\mathcal{N}}
\newcommand{\cS}{\mathcal{S}}
\newcommand{\cP}{\mathcal{P}}
\newcommand{\cH}{\mathcal{H}}
\newcommand{\triangleL}{\triangleleft_{\mathcal{L}}}
\newcommand{\triangleR}{\triangleright_{\mathcal{R}}}
\newcommand{\trangleR}{\triangleleft_{\mathcal{R}}}
\newcommand{\trangleL}{\triangleright_{\mathcal{L}}}
\newcommand{\triangleH}{\triangleleft_{\mathcal{H}}}

\newcommand{\triangleLq}{\trianglelefteq_{\mathcal{L}}}
\newcommand{\triangleRq}{\trianglelefteq_{\mathcal{R}}}
\newcommand{\trangleRq}{\trianglerighteq_{\mathcal{R}}}

\title{{\bf On Properties of Nests: Some Answers and Questions.}}

\author{Kyriakos Papadopoulos\\
\small{Kuwait University, Kuwait}}

\date{}
\maketitle

\begin{abstract}
By considering nests on a given space, we explore order-theoretical and
topological properties that are closely related to the structure of a nest.
In particular, we see how subbases given by two dual nests can be an indicator
of how close or far are the properties of the space from the structure
of a linearly ordered space. Having in mind that the term interlocking
nest is a key tool to a general solution of the orderability problem, we give a characterization of interlocking
nest via closed sets in the Alexandroff topology and via lower sets, respectively.
We also characterize bounded subsets of a given set in terms of nests and, finally, we
explore the possibility of characterizing topological groups via
properties of nests. All sections are followed
by a number of open questions, which may give new directions to the orderability
problem.

\end{abstract}



\newtheorem{definition}{Definition}[section]
\newtheorem{remark}{Remark}[section]
\newtheorem{remarks}{Remarks}[section]
\newtheorem{notation}{Notation}[section]
\newtheorem{examples}{Examples}[section]
\newtheorem{example}{Example}[section]
\newtheorem{theorem}{Theorem}[section]
\newtheorem{proposition}{Proposition}[section]
\newtheorem{co}{Corollary}[section]
\newtheorem{lemma}{Lemma}[section]
\newtheorem{corollary}{Corollary}[section]

\section{Introduction.}
\subsection{Some Motivation.}

Let $(X,\cT)$ be a topological space and let $X$ be equipped with an order relation $<$. Under
what conditions will $\cT_<$, i.e. the topology induced by the order $<$, be equal to $\cT$?
There were special solutions but no general solution to this problem, known as the {\em orderability problem}, until the early 70s.
This solution came by J. van Dalen and E. Wattel, in \cite{DalenWattel}, where the authors
used the notion of nest and interlocking family of sets. The authors of \cite{Good-Papadopoulos} expanded
the properties used in \cite{DalenWattel}, in order to characterize ordinals. The author of \cite{WillBrian}
investigated properties of nests as subbases generating topologies and the author of \cite{KyriakosPapadopoulos}
used nests in order to state a generalization of the orderability problem.
In this paper our aim is to focus on the algebraic properties of nests and investigate
how these affect the order-theoretic and topological structure of a space. In some cases
we do this in the form of questions and answers. Some open questions are rising from this study
our highlighted.

\subsection{Preliminaries.}

We first state the definitions of upper and lower
sets, which will lead to the definitions of upper
and lower topology (see \cite{Compendium}).

\begin{definition}
Let $(X,<)$ be a partially ordered set and $A \subset X$. We define $\uparrow{A} \subset X$, to be the set:
\[\uparrow{A} = \{x : x \in X \textrm{ and  there exists } y \in A \textrm{, such that } y < x\}.\]
We also define $\downarrow{A} \subset X$, to be the set:
\[\downarrow{A} = \{x : x \in X \textrm{ and  there exists } y \in A \textrm{, such that } x < y\}.\]
\end{definition}

More specifically, if $A= \{y\}$, then:
\[\uparrow{A} = \{x : x \in X \textrm{ and } y < x\}\] and
\[\downarrow{A} = \{x : x \in X \textrm{ and } x < y\}.\]

\begin{definition}
Let $X$ be a partially ordered set. A subset $A \subset X$ is a {\em lower set}, if $A = \downarrow{A}$ and an
{\em upper set}, if $A = \uparrow{A}$.
\end{definition}

The notion of upper set will be particularly useful for the characterization of the term
interlocking nest via the Alexandroff topology.

We recall
that the {\em upper topology} $\cT_U$ is generated
by the subbase $\mathcal{S} = \{X- \downarrow x : x \in X\}$
and the {\em lower topology} $\cT_l$ is generated by the subbase $\mathcal{S} = \{X-\uparrow x : x \in X\}$.
The {\em interval topology} $\cT_{in}$, on $X$, is defined as $\cT_{in}= \cT_U \vee \cT_l$, where $\vee$ stands
for the supremum (supremum in the sense that the set of all topologies on $X$ forms a complete lattice under inclusion).

\begin{definition}\label{definition-T0-T1}
Let $X$ be a set.
\begin{enumerate}
\item A collection $\mathcal{L}$, of subsets of $X$, $T_0$-{\em separates} $X$, if and only if for all $x,y \in X$, such that $x \neq y$, there exist $L \in
\mathcal{L}$, such that $x \in L$ and $y \notin L$ or $y \in L$
and $x \notin L$.

\item A collection $\mathcal{L}$, of subsets of $X$, $T_1$-{\em separates} $X$, if and only if for all $x,y \in X$, such that $x \neq y$, there exist $L,L' \in
\mathcal{L}$, such that $x \in L$ and $y \notin L$ and also $y \in
L'$ and $x \notin L'$.
\end{enumerate}
\end{definition}

One can easily see the link
between Definition \ref{definition-T0-T1} and the $T_0$ and $T_1$ separation axioms of topology: a topological space $(X,\mathcal{T})$ is
$T_0$ (resp. $T_1$), if and only if there is a subbase $\mathcal{S}$, for $\mathcal{T}$, which $T_0$-separates (resp. $T_1$-separates) $X$.

\begin{definition}
Let $X$ be a set and let $\mathcal{L}$ be a family
of subsets of $X$. $\mathcal{L}$ is a {\em nest}
on $X$, if for every $M, N \in \mathcal{L}$, either $M \subset N$ or $N \subset M$.
\end{definition}

\begin{definition}\label{definition - interlocking}
Let $X$ be a set and let $\mathcal{S} \subset \mathcal{P}(X)$. We say that
$\mathcal{S}$ is interlocking, if and only if
for each $T \in \mathcal{S}$, such that:
\[T = \bigcap \{S : T \subset S, S \in \mathcal{S} - \{T\}\}\] we have that:
\[T = \bigcup \{S : S \subset T, S \in \mathcal{S} - \{T\}\}.\]
\end{definition}

The notion of interlocking nest has played an important role to the general solution of
the orderability problem (see \cite{Good-Papadopoulos}).

\begin{definition}\label{definition-order}
Let $X$ be a set and let $\mathcal{L}$ be a nest on $X$.
We define an order relation on $X$ via the
nest $\mathcal{L}$, as follows:
\[x \triangleleft_{\mathcal{L}} y \Leftrightarrow \,\exists \, L \in \mathcal{L}, \textrm{ such that } x \in L \textrm{ and } y \notin L\]
\end{definition}

It follows from Definitions \ref{definition-T0-T1} and \ref{definition-order} that if the nest $\cL$ is $T_0$-separating, then the ordering $\triangleL$ is linear, provided the ordering is reflexive.

\begin{remark}
If $\cL$ is a nest on $X$, then by $\cL^c$ we will denote the nest $\cL^c = \{X-L : L \in \cL\} $.

We remark that $\triangleleft_{\cL^c} = \triangleright_{\cL}$, that is $x \triangleleft_{\cL^c} y$,
if and only if $y \triangleL x$.
\end{remark}

\section{An Order Relation on a set $X$ from $\mathcal{S} \subset \mathcal{P}(X)$.}

In this section we use a ``magnifying glass'' to explore
the algebraic properties of the order that was used in
\cite{DalenWattel}, where the authors gave a general
characterization of linearly ordered spaces.

\begin{definition}\label{definition-order definition by nest}
Let $X$ be a set and $\mathcal{S} \subset \mathcal{P}(X)$ be a family of subsets of $X$. We define a relation on $X$ as follows: \[x
\triangleleft_{\mathcal{S}} y \textrm{ iff there exists } S \in
\mathcal{S}, \textrm{ such that } x \in S \textrm{ and } y \notin
S\]
\end{definition}

This relation was particularly studied in \cite{Good-Papadopoulos} in the case
where $\mathcal{S}$ is a nest.

\begin{definition}
Let $X$ be a nonempty set and $<$ be an order relation, on $X$. We
say that $<$ is a {\em generated order}, if and only if there
exists $\mathcal{S} \subset \mathcal{P}(X)$, such that
$\triangleleft_{\mathcal{S}} = <$.
\end{definition}

\begin{lemma}
Let $X$ be a set and let $x,y \in X$, such that $x \neq y$. Let also $\mathcal{L} \subset \mathcal{P}(X)$ be a family
of subsets of $X$. Then, the generated order $\triangleleft_{\mathcal{L}}$, by $\mathcal{L}$,
can be defined as a subset of the Cartesian product, in the following
way:
\[\triangleleft_{\mathcal{L}} = \bigcup \{L \times (X-L) : L \in {\mathcal{L}}\}\]
\end{lemma}
\begin{proof}
Indeed, if $(x,y) \in \bigcup \{L \times (X-L) : L \in
{\mathcal{L}}\}$, then there exists $L \in {\mathcal{L}}$, such
that $(x,y) \in L \times (X-L)$.
\end{proof}

\begin{definition}
Let $X$ be a nonempty set. If $A,B \subset X \times X$, then the
composition of the two sets $A$ and $B$ is defined as follows:
\[A \circ B = \{(x,y) \in X \times X : \,\exists\,z \in X,\,(x,z) \in B \textrm{ and } (z,y) \in A\}\]
\end{definition}

\begin{proposition}\label{theorem-order pure set}
Let $X$ be a set and let $\mathcal{S} \subseteq
\mathcal{P}(X)$ be a subset of $\mathcal{P}(X)$, such that
it satisfies the property that, for every $S,T \in \mathcal{S}$, there
exists $R \in \mathcal{S}$, such that:
\[[S \times (X-S)]\,\circ [T \times (X-T)] \, \subset R \times (X-R)\]

Then,
the relation $\triangleleft_{\mathcal{S}}$ is a transitive
relation.
\end{proposition}
\begin{proof}
Let $x,y,z$ be distinct elenments in $X$, where $x
\triangleleft_{\mathcal{S}} y$ and $y \triangleleft_{\mathcal{S}}
z$. Then, there exists $S_1 \in \mathcal{S}$, such that $x \in
S_1$ and $y \notin S_1$ and also there exists $S_2 \in
\mathcal{S}$, such that $y \in S_2$ and $z \notin S_2$. This
implies that:
\[(x,y) \in S_1 \times (X-S_1)\] and
\[(y,z) \in S_2 \times (X-S_2).\]
Thus, $(x,z) \in [S_2 \times (X-S_2)] \circ [S_1 \times (X-S_1)]$.

By the hypothesis, there exists $T \in \mathcal{S}$, such that
$[S_2 \times (X-S_2)] \circ [S_1 \times (X-S_1)] \subset T \times
(X-T)$. That is, $(x,z) \in T \times (X-T)$, which implies that $x
\in T$ and $z \notin T$, which finally implies that $x
\triangleleft_{\mathcal{S}} z$, which finishes the proof.
\end{proof}

\begin{remark}
The converse of the Proposition \ref{theorem-order pure set}
is not always true. We give a counterexample. Let us consider
$\mathcal{S} = \{\{x\}: x \in X\}$. The relation
$\triangleleft_{\mathcal{S}}$ is obviously a transitive one, but
we will prove that $\mathcal{S}$ does not satisfy the condition of Proposition \ref{theorem-order pure set}. For
this, let $S \in \mathcal{S}$, where $S = \{x\} \times (X-\{x\})$
and let $T \in \mathcal{S}$, where $T = \{y\} \times (X - \{y\})$,
where $x \neq y$. We also consider $R \in \mathcal{S}$, where $R
= \{z\} \times (X - \{z\})$, such that $S \circ T \subset R$. Then,
we observe that $(y,x) \in T$ and $(x,y) \in S$. So, $(y,y) \in S
\circ T$, which also gives that $(y,y) \in R$. Thus, $y = z$ and
$y \neq z$, which leads into a contradiction.
\end{remark}

\begin{corollary}\label{corollary-transitivity}
Let $X$ be a set and let $\mathcal{S}$ be a nest on $X$. Then, the relation
$\triangleleft_{\mathcal{S}}$ is transitive.
\end{corollary}
\begin{proof}
We prove that $\mathcal{S}$ satisfies the property that for every $S,T \in \mathcal{S}$, there
exists $R \in \mathcal{S}$, such that:
\[[S \times (X-S)]\,\circ [T \times (X-T)] \, \subset R \times (X-R)\]

Let $S \in \mathcal{S}$
and $T \in \mathcal{S}$. Since $\mathcal{S}$ is a nest, we have that either $S
\subset T$ or $T \subset S$.

Let us suppose that $S \subset T$. Then, obviously, $[T \times
(X-T)] \circ [S \times (X-S)] \subset T \times (X-T)$. Indeed, let
$(x,y) \in [T \times (X-T)] \circ [S \times (X-S)]$. Then, there
exists $z \in X$, such that $(x,z) \in S \times (X-S)$ and $(z,y)
\in T \times (X-T)$. Then, $(x,z) \in S \times (X-S)$ implies that
$x \in S \subset T$, which implies that $x \in T$. But, $(x,y) \in
T \times (X-T)$, which implies that $y \notin T$. Thus, finally,
$(x,y) \in T \times (X-T)$, which completes the proof, according
to the Proposition \ref{theorem-order pure set}.
\end{proof}

\begin{proposition}\label{proposition-T_0-separation}
Let $X \neq \emptyset$ and let $\mathcal{S} \subset
\mathcal{P}(X)$. Let $\Delta = \{(x,x) : x \in X\}$. $\mathcal{S}$
$T_0$-separates $X$, if and
only if: \[X \times X - \Delta \subset \bigcup_{S \in \mathcal{S}} [S \times (X-S)] \cup [(X-S) \times S]~~~(1)\]
\end{proposition}
\begin{proof}
Let $(1)$ hold. We prove that the relation
$\mathcal{S}$ $T_0$-separates $X$. So, let $x,y \in
X$, where $x \neq y$. Then, there exists $S \in \mathcal{S}$,
such that $(x,y) \in S \times (X-S)$ or there exists $T \in \mathcal{S}$,
such that $(x,y) \in (X-T) \times T$. So, either $x \in S$ and $y \notin S$ or
$y \in T$ and $x \notin T$. Thus, $\mathcal{S}$ is $T_0$-separating.

On the other hand, we consider $\mathcal{S}$
to be $T_0$-separating and we will prove that $(1)$ holds. Indeed,
if $(x,y) \in X \times X - \Delta$, then for $x \neq y$, and since
$\mathcal{S}$ is $T_0$-separating, $x \triangleleft_{\mathcal{S}} y$ or $x
\triangleright_{\mathcal{S}} y$.

If $x \triangleleft_{\mathcal{S}} y$, then there exists $S \in
\mathcal{S}$, such that $x \in S$ and $y \notin S$, which implies
that there exists $S \in \mathcal{S}$, such that $(x,y) \in
S \times (X-S)$.

If $x \triangleright_{\mathcal{S}}  y$, then there exists $S \in
\mathcal{S}$, such that $y \in S$ and $x \notin S$, which implies
that $x \in X-S = T$, say, and $y \notin T$, where $X-T \in
\mathcal{S}$, which implies that $(x,y) \in T \times (X-T)$.

So, $(1)$ holds, and this completes the proof.
\end{proof}

An immediate consequence of Proposition \ref{proposition-T_0-separation}
is the following corollary.

\begin{corollary}
A topological space $X$ is $T_0$, if and only if it admits a subbasis
$\mathcal{S}$, such that: \[X \times X - \Delta \subset \bigcup_{S \in \mathcal{S}} [S \times (X-S)] \cup [(X-S) \times S]\]
\end{corollary}

\begin{proposition}
Let $X$ be a nonempty set and let also
$\mathcal{S}_1,\mathcal{S}_2 \subset \mathcal{P}(X)$, such that
$\emptyset \in \mathcal{S}_1$ and $\emptyset \in \mathcal{S}_2$.
We define:
\[\mathcal{S}_1 \cup^{\star} \mathcal{S}_2 = \{S_1 \cup S_2 : S_1 \in \mathcal{S}_1,\,S_2 \in \mathcal{S}_2\}\]
Then, $\triangleleft_{\mathcal{S}_1 \cup^{\star} \mathcal{S}_2} =
\triangleleft_{\mathcal{S}_1} \cup \triangleleft_{\mathcal{S}_2}$.
\end{proposition}
\begin{proof}
Let $x \triangleleft_{\mathcal{S}_1 \cup^{\star} \mathcal{S}_2}
y$. Then, there exist $S_1 \cup S_2 \in \mathcal{S}_1
\bigcup^{\star} \mathcal{S}_2$, where $S_1 \in \mathcal{S}_1$ and
$S_2 \in \mathcal{S}_2$, such that $x \in S_1 \cup S_2$ and $y
\notin S_1 \cup S_2$, something that implies that $[x \in S_1
\textrm{ and } y \notin S_1]$ or $[x \in S_2 \textrm{ and } y
\notin S_2]$, which implies that $x \triangleleft_{\mathcal{S}_1}
y$ or $x \triangleleft_{\mathcal{S}_2} y$, which finally implies
that $(x,y) \in \triangleleft_{\mathcal{S}_1} \cup
\triangleleft_{\mathcal{S}_2}$, as required.

On the other way round, let $(x,y) \in
\triangleleft_{\mathcal{S}_1} \cup \triangleleft_{\mathcal{S}_2}$.
But this implies that $x \triangleleft_{\mathcal{S}_1} y$ or $x
\triangleleft_{\mathcal{S}_2} y$, which implies that there exists
$S_1 \in \mathcal{S}_1$, such that $[x \in S_1 \textrm{ and } y
\notin S_1]$ or $[x \in S_2 \textrm{ and } y \notin
\mathcal{S}_2]$, which implies that $(x \in S_1 \cup \emptyset
\textrm{ and } y \notin S_1 \cup \emptyset )$ or $(x \in \emptyset
\cup S_2 \textrm{ and } y \notin \emptyset \cup S_2)$.

In each case, we remark that $x \triangleleft_{\mathcal{S}_1
\cup^{\star} \mathcal{S}_2} y$. Thus,
$\triangleleft_{\mathcal{S}_1} \cup \triangleleft_{\mathcal{S}_2}
\, \subset \, \triangleleft_{\mathcal{S}_1 \cup^{\star}
\mathcal{S}_2}$.

Finally, $\triangleleft_{\mathcal{S}_1 \cup^{\star} \mathcal{S}_2}
= \triangleleft_{\mathcal{S}_1} \cup
\triangleleft_{\mathcal{S}_2}$.
\end{proof}

\begin{remark}
Let $X$ be a nonempty set and $\mathcal{S},\,\mathcal{S}' \in
\mathcal{P}(X)$. We define an equivalence relation, as follows:
\[\mathcal{S} \sim \mathcal{S}' \Leftrightarrow \triangleleft_{\mathcal{S}} = \triangleleft_{\mathcal{S}'}\]
So, the equivalence class of an arbitrary $\mathcal{S} \subset
\mathcal{P}(X)$ will be of the form:
\[\bf{C}(\mathcal{S}) = \{\mathcal{S}' \subset \mathcal{P}(X) : \triangleleft_{\mathcal{S}'} = \triangleleft_{\mathcal{S}}\}\]
\end{remark}

\begin{remark}
Let $X \neq \emptyset$ and $A \subset X$. Let also
$\mathcal{S} = \{A\}$. Then, $\triangleleft_{\mathcal{S}} = A
\times (X - A)$ and $\bf{C}(\mathcal{S}) = \{\mathcal{S}\}$.

We also remark that if we take $\mathcal{S} = \{w\}$, where $w \in
X$, then $x \triangleleft_{\mathcal{S}} y$ is equivalent to $x = w$
and $y \neq w$. Thus, $w$ is the minimal element of $X$, with respect
to $\mathcal{S}$.
\end{remark}

\begin{example}
Let $X$ be an infinite set and let $\mathcal{S} = \{S \subset X,
\textrm{ where } X-S \textrm{ is finite}\}$. Take $x,y \in X$, such that $x \neq y$. We
consider $S_1 = X-\{y\}$ and $S_2 = X - \{x\}$. Then, $x \in S_1$
and $y \notin S_1$ implies that $x \triangleleft_{\mathcal{S}} y$.
Similarly, $y \in S_2$ and $x \notin S_1$ implies $x
\triangleleft_{\mathcal{S}} y$. Thus, $x
\triangleleft_{\mathcal{S}} y$ is not an antisymmetric relation,
but it is obviously a transitive one.
\end{example}

\begin{example}
We claim that the usual order $<$, on $\mathbb{R}$, is a generated
order. Indeed, let $\mathcal{S} = \{(-\infty,x): x \in
\mathbb{R}\}$. Then, $< =
\triangleleft_{\mathcal{S}}$.
\end{example}
\begin{proof}
We first prove that $\triangleleft_{\mathcal{S}} \subset <$. Let
$a \triangleleft_{\mathcal{S}} b$. Then, there exists $(-\infty,x)
\in \mathcal{S}$, such that $a \in (-\infty,x)$ and $b \notin
(-\infty,x)$. Then, $b \in [x,\infty)$. So, $a<b$, which implies
that $\triangleleft_{\mathcal{S}} \subset <$.

Conversely, let $a<b$. Then, we take $x \in (a,b)$. So, $a \in
(-\infty,x)$ and $b \notin (-\infty,x)$, which implies that there
exists $S = (-\infty,x) \in \mathcal{S}$, such that $a \in S$ and
$b \notin S$, which implies $a \triangleleft_{\mathcal{S}} b$,
which implies $< \subset \triangleleft_{\mathcal{S}}$.

Finally, $< = \triangleleft_{\mathcal{S}}$.
\end{proof}

\section{Nests and Orders: Some Further Remarks.}

\subsection{$T_0$-separating Nests as a Measure of Linearity.}

Consider the set of real numbers $\mathbb{R}$, equipped with its usual topology. Let
$\cL = \{(-\infty,a): a \in \mathbb{R}\}$. We remark that for each $(-\infty,a)\in \cL$,
$\sup L = a \notin L$. We also remark that for each $k \in \mathbb{R}$, there exists $L=(-\infty,k) \in \cL$,
such that $\sup L = k$. We will now generalise this remark to arbitrary sets.
In particular, we will use the following three conditions, namely (C1), (C2), (C3),
in order to investigate the relationship between the topologies $\cT_{\cL \cup \cR}$ and
$\cT_{in}^\cL$; this relationship will be a measure of linearity, that is, it will
show how close -or not- is a space from a LOTS, regarding its structure.
 From now on, $\sup$ will be used for abbreviating the term {\em supremum}
and $\inf$ will abbreviate the term {\em infimum}.

Let $\cL$ be a nest on a set $X$. We introduce the following three conditions:

\begin{enumerate}

\item[(C1)] For each $L \in \cL$, there exists $\sup L$ with respect $\triangleLq$.

\item[(C2)] For each $L \in \cL$, there exists $\sup L$ with respect to $\triangleLq$, such that $\sup L \in X-L$.

\item[(C3)] For each $x$, there exists $L \in \cL$, such that there exists $\sup L=x \in X-L$ and also
property (C2) holds.

\end{enumerate}

We deduce the following relations between (C1), (C2) and (C3).

\begin{proposition}\label{proposition-properties}
~
\begin{enumerate}

\item (C3) implies (C2).

\item (C2) implies (C1).

\item (C1) does not always imply (C2).

\item (C2) does not always imply (C3).

\item (C3) implies that $\cL$ is $T_0$-separating.

\item $\cL$ $T_0$-separating implies neither (C1) nor (C2) nor (C3).

\item Neither (C1) nor (C2) imply that $\cL$ is $T_0$-separating.

\end{enumerate}
\end{proposition}
\begin{proof}
The statement that (C3) implies (C2) follows immediately from the definition of (C3).
Similarly, (C2) implies (C1) by the definition of (C2). Example \ref{example-(]}
shows that (C1) does not always imply (C2) or $T_0$-separation. Example \ref{example-()}
shows that (C2) does not always imply (C3) or $T_0$-separation. Proposition \ref{proposition-C3 implies T0}
shows that (C3) implies $T_0$-separation. Examples \ref{example-T0 not nec C1}, \ref{example-T0 not nec C2} and
\ref{example-T0 not nec C3} show that the $T_0$-separation of $\cL$ does not necessarily imply
property (C1) or (C2) or (C3).
\end{proof}


\begin{example}\label{example-(]}
Let $X=(0,1)$ and consider the nest $\cL = \{(0,a]: a \in \mathbb{R}, {{1}\over{2}} \le a <1\}$, on $X$. We remark
that condition (C1) is satisfied, but (C2) is not satisfied. This is because for each $L \in \cL$, $\sup L = a \in L$.
This counterexample shows that (C1) does not always imply (C2). We also see that $\cL$ is not $T_0$-separating,
because there does not exist $L \in \cL$ that $T_0$-separates, say, ${{1}\over{4}}$ and ${{1}\over{8}}$. This shows
that condition (C1) does not always imply $T_0$-separation.
\end{example}

\begin{example}\label{example-()}
Let $X= (0,1)$ and consider the nest $\cL = \{(0,a): a \in \mathbb{R}, {{1}\over{2}} \le a <1\}$, on $X$. We remark
that condition (C2) is satisfied, but condition (C3) is not satisfied. This is because for each
$L \in \cL$, $\sup L = a \notin L$; this shows that (C2) is satisfied. But we also
see that there does not exist $L \in \cL$, such that $\sup L = {{1}\over{4}} \in X-L$. This counterexample
shows that (C2) does not always imply (C3) and also (C2) does not always imply that $\cL$ is $T_0$-separating. Indeed,
there does not exist $L \in \cL$ that $T_0$-separates ${{1}\over{4}}$ and ${{1}\over{8}}$.
\end{example}

\begin{remark}
The results in both Examples \ref{example-()} and \ref{example-(]} permit us to make some conclusions
on the connection between $T_0$-separating nests and linear orders. It follows from the definition
of nest and $T_0$-separation that
a nest is $T_0$-separating, if and only if $\triangleLq$ is a linear order.
Why isn't the nest $\cL$, in both of the above examples of subsets of the real line, \ref{example-()} and \ref{example-(]},
$T_0$-separating? The answer lies on the fact that the elements of the nest $\cL$
must satisfy a bijection with the elements of the set $X$, something that does not happen in our examples. So,
the set $X$, in Examples \ref{example-()} and \ref{example-(]} is not linearly ordered via $\triangleLq$.
\end{remark}

\begin{proposition}
Let $<$ be a linear order on a set $X$ and let $\cL_< = \{(-\infty,a) : a \in X\}$ be a nest on $X$, such
that  Let $|\cL| = |X|$. Then, $\cL$ $T_0$-separates $X$.
\end{proposition}

\begin{example}\label{example-T0 not nec C3}
Let $X = \{a,b\}$ and consider the nest $\cL = \{\{a\}\}$, on $X$. We remark that $\cL$ is $T_0$-separating.
Indeed, since $a \neq b$, there exists $L = \{a\} \in \cL$, such that $a \in \{a\}$ and $b \notin \{a\}$. We
remark that (C3) is not satisfied though. Indeed, for $b \in X$, there does not exist $L \in \cL$, such that $\sup L = b$.
We observe that $L = \{a\} \in \cL$ and that $\sup L = a$.
\end{example}

\begin{example}\label{example-T0 not nec C2}
Consider $X = \mathbb{R}$ and the nest $\cL = \{(-\infty,a]: a \in \cR\}$, on $\mathbb{R}$. One
can easily see that $\cL$ $T_0$-separates $\mathbb{R}$. But, for each $L \in \cL$, we have that
$\sup (-\infty,a]=a \in L$. So, property (C2) is not satisfied. With this example we see that the $T_0$-separation
property of $\cL$ does not necessarily imply property $(C2)$.
\end{example}

\begin{example}\label{example-T0 not nec C1}
Let $X=\mathbb Q$.  For each $r\in \mathbb R$, let $L_r=(-\infty,r)\cap X$ and let $\mathcal L=\{L_r:r\in \mathbb R\}$.  Certainly $\mathcal L$ is $T_0$-separating and $\mathcal L$ generates the usual order on $\mathbb Q$.  But $L_{\sqrt 2}$ does not have a supremum in $X$.
\end{example}

We will now prove that property (C3) implies the $T_0$-separation of $\cL$.

\begin{proposition}\label{proposition-C3 implies T0}
Let $X$ be a set and let $\cL$ be a nest on $X$ that satisfies property (C3). Then, $\cL$ $T_0$-separates $X$.
\end{proposition}
\begin{proof}
Let $x,y \in X$, such that $x \neq y \in X$. By (C3), there exists $L_x \in \cL$, such that $\sup L_x = x$ and there also
exists $L_y \in \cL$, such that $\sup L_y = y$. Since $\cL$ is a nest on $X$, we have that either $L_x \subset L_y$
or $L_y \subset L_x$. If $L_x \subset L_y$, then $\sup L_x \triangleLq \sup L_y$, which implies that $x \triangleL y$.
If $L_y \subset L_x$, we have that $\sup L_y \triangleLq \sup L_x$, which implies that $y \triangleL x$. So, either $x \triangleL y$ or $y \triangleL x$, proving that $\cL$ $T_0$-separates $X$.
\end{proof}

\begin{lemma}\label{lema-lemmas 1 and 2}
Let $X$ be a set and let $\cL \subset \cP(X)$ be a nest.

\begin{enumerate}
\item If condition (C1) is satisfied and $\sup L = k$,
then $L \supset X-\uparrow{k}$.

\item If condition (C2) is satisfied and $\sup L = k$,
then $L \subset X-\uparrow{k}$.

\end{enumerate}
\end{lemma}
\begin{proof}
1. Let $L \in \cL$ and let $k = \sup L \in X$. Then, for each $x \in L$, $x \triangleLq k$.
Let $y \in X-L$. Since $x \in L$ and $y \notin L$, we have that $x \triangleL y$, for each $x$. So, $k \triangleLq y$, and so $y \in \uparrow{k}$.
Thus, for each $y \in X-L$, we have that $y \in \uparrow{k}$. The latter gives that $X-L \subset \uparrow{k}$,
which implies that $L \supset X- \uparrow{k}$.

2. For each $x \in L$, we have $x \triangleL k$, so $k \ntrianglelefteq x$ \footnote{Indeed, if $k=x$ we get a contradiction.
If $x \triangleL k$, then there exists $L_1 \in \cL$, such that $x \in L_1$ and $k \notin L_1$. If $k \triangleL x$,
then there exists $L_2 \in \cL$, such that $k \in L_2$ and $x \notin L_2$. But $\cL$ is a nest. If $L_1 \subset L_2$,
 then $x \notin L_1$ and $x \in L_1$, a contradiction. If $L_2 \subset L_1$ we get a contradiction in a similar way.}, which implies that $x \in X- \uparrow{k}$.
Thus, $L \subset X - \uparrow{k}$.
\end{proof}

From now on, $\cT_{\cL}$ will denote the topology generated by the nest $\cL$, on $X$, and
$\cT_l$ the lower topology on $X$.

\begin{proposition}\label{proposition-lemmas 1 and 2}
Let $X$ be a set and let $\cL \subset \cP(X)$ be a nest. If condition (C2) is satisfied, then:
\begin{enumerate}

\item $L = X-\uparrow{k}$, where $k = \sup L$, with respect to $\triangleL$, for each $L \in \cL$.

\item $\cT_{\cL} \subset \cT_l$.
\end{enumerate}
\end{proposition}
\begin{proof}
1. follows by Lemma \ref{lema-lemmas 1 and 2}.

2. $\cT_l$ is of the form $\mathcal{S} = \{X-\uparrow{k}: k \in X\}$. Let $L \in \cL$. Part 1.
gives that $L = X - \uparrow{k}$, so $L \in \cT_l$ and the result follows.
\end{proof}

\begin{theorem}\label{theorem-1 and 2}
Let $X$ be a set and let $\cL \subset \cP(X)$ be a nest on $X$, such that condition (C3) is satisfied. Then, $\cT_{\cL} = \cT_l$.
\end{theorem}
\begin{proof}
Proposition \ref{proposition-lemmas 1 and 2} gives that $\cT_{\cL} \subset \cT_l$. We now
consider a subbasic open set of $\cT_l$ of the form $X- \uparrow{x}$. Then, there exists
$L \in \cL$, such that $\sup L = x$. But, according to Proposition \ref{proposition-lemmas 1 and 2}, $L = X - \uparrow{x}$.
So, $\cT_{\cL} \subset \cT_l$ and the statement of the theorem follows.
\end{proof}

\begin{remark}\label{remark-on L and R}
Let $\cL$ be a nest on a set $X$. Let $\cR$ be another nest on $X$, such that there exists
a mapping from $\cL$ to $\cR$, so that $x \triangleL y$, if and only
if $y \trangleR x$. So, $x \triangleL y$, if and only if there exists $L \in \cL$, such that
$x \in L$ and $y \notin L$, if and only if there exists $R \in \cR$, such that $y \in R$ and $x \notin R$.

Note that we do not demand from $\cL \cup \cR$
to form a $T_1$-separating subbase for $X$;  so neither $\cL$ nor $\cR$
will necessarily $T_0$-separate $X$. We keep only the dual order-theoretic properties of these two nests, but
we do not necessarily keep the property that restricts them on a line. So, we are now able to rewrite for $\cR$,
in a dual way, the properties that hold for $\cL$.
\end{remark}

\begin{definition}\label{definition-binary}
Let $X$ be a set and let $\cL$ and $\cR$ be two nests on $X$, that satisfy the properties
of Remark \ref{remark-on L and R}. We call such nests {\em dual nests}. $\cL$ will be
called {\em dual to} $\cR$ and $\cR$ dual to $\cL$.
\end{definition}

Let $X$ be a set and let $\cR$ be dual to the nest $\cL$, where $\cL$ satisfies
properties (C1),(C2),(C3). In a similar fashion, we define the following properties
for $\cR$:

\begin{enumerate}

\item[(C1)*] For each $R \in \cR$, there exists $\sup R$ with respect to $\trangleRq$. \\
(Equivalently, for each $R \in \cR$, there exists $\inf R$ with respect to $\triangleLq$.)

\item[(C2)*] For each $R \in \cR$, there exists $\sup R$ with respect to $\trangleRq$, such that $\sup R \in X-R$.\\
(Equivalently, for each $R \in \cR$ there exists $\inf R$ with respect to $\triangleLq$, such that $\inf R \in X-R$).

\item[(C3)*] For each $x \in X$, there exists $R \in \cR$, such that there exists $\sup R \in X-R$ with respect to $\trangleRq$ and also property (C2)* holds.\\
(Equivalently, for each $x \in X$, there exists $R \in \cR$, such that there exists $\inf R \in X-R$, with respect to $\triangleLq$ and also property (C2)* holds).

\end{enumerate}

One easily observes that Proposition \ref{proposition-properties} holds, too, if we substitute
(C1)*, (C2)*, (C3)* in the place of (C1), (C2),(C3), respectively.

Proposition \ref{proposition-lemmas 1 and 2} can be also stated with respect to $\cR$ in a dual way.

\begin{proposition}\label{binary proposition-lemmas 1 and 2}
Let $X$ be a set and let $\cR \subset \cP(X)$ be a nest. If condition (C2)* is satisfied, then:
\begin{enumerate}

\item $R = X-\uparrow{k}$, where $k = \sup R$ with respect to $\trangleRq$ for each $R \in \cR$ (or, equivalently,
$R= X-\downarrow{k}$, where $k= \inf R$ with respect to $\triangleLq$).

\item $\cT_{\cR} \subset \cT_U$.
\end{enumerate}
\end{proposition}

In a similar way, we can restate Theorem \ref{theorem-1 and 2}, with respect to $\cR$.

\begin{theorem}\label{binary theorem-1 and 2}
Let $X$ be a set and let $\cR \subset \cP(X)$ be a nest on $X$, such that condition (C3)* is satisfied. Then $\cT_{\cR} = \cT_U$.
\end{theorem}

We can now sum up Theorems \ref{theorem-1 and 2} and \ref{binary theorem-1 and 2}, in the
following theorem.

\begin{theorem}\label{theorem-conclusive}
Let $X$ be a set and let $\cL$ and $\cR$ be two dual nests on $X$.
\begin{enumerate}
\item If $\cL$ satisfies
(C2) and if $\cR$ satisfies (C2)*, then $\cT_{\cL \cup \cR} \subset \cT_{in}^\cL$.

\item If $\cL$ satisfies (C3) and if $\cR$ satisfies (C3)*, then $\cT_{\cL \cup \cR} = \cT_{in}^\cL$.

\end{enumerate}
\end{theorem}

As we can see in the two examples that follow, the conditions of statements 1. and 2. from Theorem \ref{theorem-conclusive}
are sufficient but not necessary.

\begin{example}
Let $X = \{x_1,x_2\}$ and let $\cL = \{\{x_1\}\}$. Then, $\cT_\cL = \{\{x_1\},\{x_1,x_2\},\emptyset\}$ is
the topology on $X$ which is generated by $\cL$. We observe that $x_1 \triangleL x_2$. Then,
$\uparrow{x_1}= \{x_1,x_2\}$, $X-\uparrow{x_1} = \emptyset$, $\uparrow{x_2}= \{x_2\}$ and $X-\uparrow{x_2}= \{x_1\}$.
So, the lower topology $\cT_l = \{\emptyset,\{x_1\},\{x_1,x_2\}\} = \cT_\cL$. Now, we define $\cR = \{\{x_2\}\}$
and $x_2 \triangleR x_1$, if and only if there exists $R \in \cR$, such that $x_2 \in R$ and $x_1 \notin R$.
So, $x_1 \triangleL x_2$ if and only if $x_2 \triangleR x_1$. Then, $\cT_\cR = \{\{x_2\},\{x_1,x_2\}\}$ is
the topology on $X$ which is induced by $\cR$. Also, $\downarrow{x_1}= \{x_1\}$, $\downarrow{x_2}= \{x_1,x_2\}$,
$X-\downarrow{x_1}= \{x_2\}$ and $X- \downarrow{x_2}= \emptyset$. So, the upper topology $\cT_U = \{\emptyset,\{x_2\},\{x_1,x_2\}\} = \cT_\cR$.

From the above, we conclude that $\cT_{\cL \cup \cR} = \cT_{in}^\cL$ is equal to the discrete topology, although property (C3) is not satisfied.
This is because $x_2$ is not the supremum of any element of $\cL$.
\end{example}

\begin{example}
Let $X = \{x_1,x_2,x_3,x_4\}$ and let $\cL = \{\{x_1,x_2\},\{x_1,x_2,x_3,x_4\}\}$. Then,
one can easily see that $x_2 \triangleL x_3$, $x_2 \triangleL x_4$, $x_1 \triangleL x_3$ and $x_1 \triangleL x_4$.
Also, $\uparrow{x_1} = \{y \in X: x_1 \triangleLq y\}  = \{x_1,x_3,x_4\}$ and $X- \uparrow{x_1}= \{x_2\}$.
Similarly, $\uparrow{x_2} = \{x_2,x_3,x_4\}$ and $X- \uparrow{x_2}= \{x_1\}$; $\uparrow{x_3} = \{x_3\}$ and
$X - \uparrow{x_3} = \{x_1,x_2,x_4\}$; $\uparrow{x_4}= \{x_4\}$ and $X- \uparrow{x_4} = \{x_1,x_2,x_3\}$.
The lower topology now takes the form $\cT_l = \{\emptyset, \{x_1\}, \{x_2\},\{x_1,x_2\}, \{x_1,x_2,x_3\},
\{x_1,x_2,x_4\},\linebreak\{x_1,x_2,x_3,x_4\}\}$ and $\cT_\cL = \{\emptyset, \{x_1,x_2\},\{x_1,x_2,x_3,x_4\}\}$. So, $\cT_\cL \subset \cT_l$,
but $\cL$ is not $T_0$-separating, because $x_3 \neq x_4$ and there is no $L \in \cL$ that $T_0$-separates $x_3$ and $x_4$.
Also, $\cL$ does not satisfy property (C2), because $\sup\{x_1,x_2\}$ does not exist.

Now, we consider $\cR = \{\{x_3,x_4\},\{x_1,x_2,x_3,x_4\}\}$, and we observe that $x_3 \triangleR x_1, x_3 \triangleR x_2,
x_4 \triangleR x_2$ and $x_4 \triangleR x_3$. So, there exists a mapping between the
nests $\cL$ and $\cR$, and their duality can be seen from the fact that $x_3 \triangleR x_1$ iff $x_1 \triangleL x_3$,
$x_3 \triangleR x_2$ iff $x_2 \triangleL x_3$, $x_4 \triangleR x_2$ iff $x_2 \triangleL x_4$ and $x_4 \triangleR x_1$
iff $x_1 \triangleL x_4$. It can be easily deduced that $\cT_\cR = \{\emptyset, \{x_3,x_4\}, \{x_1,x_2,x_3,x_4\}\}$
and that the upper topology is $\cT_U  = \{\emptyset, \{x_2,x_3,x_4\}, \{x_1,x_3,x_4\},\{x_3\},\{x_4\},
\linebreak\{x_3,x_4\},\{x_1,x_2,x_3,x_4\}\}$. Also,
$\cR$ is not $T_0$-separating, neither satisfies property (C2)* and we deduce that $\cT_\cR \subset \cT_U$.
Last, but not least, we see that $\cT_{in}^\cL$ is the discrete topology, thus $\cT_{\cL \cup \cR} \subset \cT_{in}^{\cL}$.
\end{example}

We have stated that a
non-reflexive order that is induced by a nest $\cL$ makes $\cT_{in}^\cL$ equal to
the discrete topology, so it will automatically be finer than $\cT_{\cL \cup \cR}$. If the
order is reflexive, then Theorem \ref{theorem-conclusive} shows that there is a case
where $\cT_{in}^\cL$ is equal to $\cT_{\cL \cup \cR}$, and this is when properties (C3) and (C3)*
are both satisfied. But (C3) (resp. (C3)*) implies that $\cL$ (resp. $\cR$) is $T_0$-separating, while in Example \ref{example-T0 not nec C3}
(and Proposition \ref{proposition-properties})
we see that $\cL$ can be $T_0$-separating, without (C3) being satisfied. So, the
two topologies coincide in certain type of spaces that are $T_0$-separating under properties (C3) and (C3)*.

The real line, with its natural topology that is generated by the nests $\cL =\{(-\infty,a): a \in \mathbb{R}\}$ and
$\cR= \{(a,\infty): a \in \mathbb{R}\} $ is a specific example of a space of the type that is described
in Theorem \ref{theorem-conclusive} 2. Question: are there other LOTS, apart from the real line with its
natural order, such that 2. from Theorem \ref{theorem-conclusive} is satisfied? The answer is positive.
Consider, for example, sum of copies of the real line. Other spaces admitting such nests are connected
orderable spaces with no minimal and maximal elements (for instance the long line).

Furthermore, we remark that if property (C2) alone is satisfied, then for each $L \in \cL$ we have
that $\sup L \in X-L$, so that $\sup L \notin L$. So, for each $L \in \cL$ there is no $\triangleLq$-maximal
element in $L$, because for each $L \in \cL$ there exists $k \in L$, $x \triangleLq k$, for each $x \in L$,
so that $k = \sup L$. In a similar fashion, we can obtain a dual property for the dual nest $\cR$, with
the order $\trangleRq$.

\begin{corollary}\label{theorem-finalorderability}
Let $X$ be a set and let $\cL,\cR$ be two nests on $X$, such that
$\triangleL = \triangleR$. Let also properties (C3) and (C3)* be satisfied. Then,
$X$ is a LOTS.
\end{corollary}
\begin{proof}
We observe that property (C3) (similarly (C3)*) implies $T_0$-separation and interlocking,
so that the conditions of van Dalen and Wattel follow immediately and so $X$ is a LOTS.
\end{proof}

Property (C3) (resp. (C3)*) implies naturally $T_0$-separation and interlocking. Property (C2) (resp. (C2)*) implies
interlocking, if we add $T_0$-separation. So, we can restate Corollary \ref{theorem-finalorderability} as follows:

\begin{corollary}\label{theorem-lots from c2}
Let $X$ be a set and let $\cL,\cR$ be two nests on $X$, such that
$\triangleL = \triangleR$ and each of $\cL$ and $\cR$ $T_0$-separates $X$, respectively. Let also properties (C2) and (C2)* be satisfied. Then,
$X$ is a LOTS.
\end{corollary}

Question:
what is the difference between LOTS that are
implied by Corollary \ref{theorem-finalorderability} from LOTS being implied by Corollary \ref{theorem-lots from c2}? The answer is that
the two corollaries claim the same result. Namely, for a nest $\cL$ of subsets of $X$, (C3) follows by (C2), provided that
$\triangleLq$ is a linear order on $X$. Indeed, suppose $\triangleLq$ is a linear order on $X$ and $\cL$ satisfies (C2).
Then, the nest $\cH = \{\{x \in X : x \triangleL y\}: y \in X\}$ satisfies (C3) and $\triangleH = \triangleL$. To see
this, take a point $y \in X$. If $y$ is the $\triangleLq$-first element of $X$, then it is the $\triangleLq$-supremum
of the empty set. Suppose there exists $x \in X$, with $x \triangleL y$ and let $H = \{x \in X : x \triangleL y\} \in \cH$.
If $y$ is not the $\triangleLq$-supremum of $H$, then $H$ has a $\triangleLq$-maximal element, namely $z$. Since $z \triangleL y$,
there exists $L_z \in \cL$, such that $z \in L_z$ and $y \notin L_z$. If $x \in H$ and $x \neq z$, it follows that
$x \triangleL z$ and, by the same reason, $x \in L_x$, for some $L_x \in \cL$ for which $z \notin L_x$. Since
$\cL$ is a nest, $x \in L_x \subset L_z$. Thus, $H \subset L_z$ and, in fact, $H = L_z$, because $y \notin L_z$.
By (C2), the $\triangleL$-supremum of $H = L_z$ does not belong to $H$, a contradiction, because $z \in H$
and $x \triangleL z$, for every $x \in H$.

The following example shows that both properties (C2) and (C2)* do not necessarily imply $T_0$-separation. So, Theorem \ref{theorem-lots from c2} without the $T_0$-separation property of $\cL$ and $\cR$ generates spaces that are not necessarily linearly ordered, but
carry analogous order theoretic properties to linearly ordered sets.

\begin{example}
Consider the set of real numbers $\mathbb{R}$ and the nests $\cL = \{(-\infty,n): n \in \mathbb{N}\}$
and $\cR = \{(n,\infty): n \in \mathbb{N}\}$ on $\mathbb{R}$. Then, $\cL$ and $\cR$ satisfy
conditions (C2) and (C2)*, respectively. Indeed, for each $L = (-\infty,n) \in \cL$, $\sup L = n \notin L$
and for each $R \in \cR$, $\inf(n,\infty)= n \notin R$. We also remark, from the definition of $T_0$-separation, that neither $\cL$ nor $\cR$ is
$T_0$-separating.
\end{example}

{\bf{Open Question 1.}} Are there conditions that can be added in Corollaries \ref{theorem-finalorderability} and \ref{theorem-lots from c2}, respectively,
so that they will lead into a brand new characterization of LOTS?

\subsection{Interlocking Nests via the Alexandroff Topology.}

It is known that, for a partially ordered set $(X,<)$, the Alexandroff
topology is the family $\mathcal{A} = \{U \subset X : U = \uparrow U\}$
(see \cite{Compendium}).

\begin{proposition}\label{proposition-for the proof}
Let $X$ be a set and let $\mathcal{L}$ be a nest on $X$. If $Y \subset X$, then:
\[
\uparrow{Y} =
\bigcup_{L \in \cL} \{X-L : Y \cap L \neq \emptyset\}.
\]
\end{proposition}

\begin{proof}
$x \in \bigcup_{L \in \cL} \{X-L : Y \cap L \neq \emptyset\}$, if and only
if there exists $L \in \cL$, such that $x \in X-L$ and $Y \cap L \neq \emptyset$, if and only if
there exists $L \in \cL$, such that $x \in X-L$ and there exists $y \in Y \cap L$, if and only
if there exists $L \in \cL$, such that $y \in Y$ where $y \in L$ and $x \notin L$, if and only
if there exists $y \in Y$, such that $y \triangleL x$, if and only if $x \in \uparrow{Y}$.
\end{proof}

\begin{proposition}\label{proposition-Alex}
Let $X$ be a set and let $\mathcal{L}$ be a nest on $X$. Then, the Alexandroff topology, on $X$, is
given by the collection:
\[\mathcal{A} = \{Y \subset X : Y = \bigcup\{X-L : Y \cap L \neq \emptyset\}\}\]
\end{proposition}
\begin{proof}
The proof follows immediately from Proposition \ref{proposition-for the proof}.
\end{proof}

\begin{proposition}\label{proposition-preparatory for final}
Let $X$ be a set and let $\cL$ be a nest on $X$. A set $M \in \cL$ is closed with
respect to the Alexandroff topology via $\triangleL$, if and only if $M$ can take the following form:
\[M = \bigcap \{L \in \cL : M \subsetneq L\}.\]
\end{proposition}
\begin{proof}
$M$ is closed with respect to the Alexandroff topology, if and only if $M^c$ is open, i.e.
$M^c = \bigcup\{X-L : L \in \cL, (X-M) \cap L \neq \emptyset\}$, if and only if
$M^c = \bigcup \{X-L : L \nsubseteq M\}$, if and only if $M^c = \bigcup \{X-L : M \subsetneq L\}$,
i.e. $M = \bigcap \{L : M \subsetneq L\}$.
\end{proof}

\begin{proposition}
Let $X$ be a set and let $\cL$ be a nest on $X$. A set $X-L \in \cL^{c}$ is closed with respect
to the Alexandroff topology which is defined by $\triangleleft_{\cL^c}$, if and only if
$M = \bigcup \{L \in \cL : L \subsetneq M\}$.
\end{proposition}

\begin{proof}
According to Proposition \ref{proposition-preparatory for final}, $X-M$ is
closed with respect to the Alexandroff topology via $\triangleleft_{\cL^c}$, if
and only if $X-M = \bigcap \{X-L: X-M \subsetneq X-L\} = \bigcap\{X-L : L \subsetneq M\}$,
if and only if $M = \bigcup \{L : L \subsetneq M\}$.
\end{proof}

\begin{theorem}\label{corollary - interl alex}
Let $X$ be a set and let $\cL$ be a nest on $X$. $\cL$ is {\em interlocking},
if and only if for each $L \in \cL$, such that $L$ is closed with respect
to the Alexandroff topology via $\triangleL$,
we have that $X-L$ is closed with respect to the Alexandroff topology via $\triangleleft_{\mathcal{L}^c}$.
\end{theorem}

We note that the characterization of Theorem \ref{corollary - interl alex} does not require
the nests to be interlocking, so it is equivalent to Definition \ref{definition - interlocking}.

{\bf{Open Question 2.}} Given the characterization of interlocking nest in terms of
closed sets of the Alexandroff topology, how can this lead into a restatement of the
orderability problem as (re-)stated in \cite{Good-Papadopoulos}? Will such a restatement
lead into a brand new proof, using purely topological (than order-theoretic) tools?

\subsection{Interlocking Nests via Lower Sets.}

\begin{proposition}\label{proposition - downset}
Let $X$ be a set, let $\mathcal{L}$ be a nest on $X$ and
$\triangleleft_{\mathcal{L}}$ be the corresponding order on $X$ that is induced
by $\mathcal{L}$. If $Y \subset X$, then:

\[\downarrow{Y} =
\bigcup\{L \in \mathcal{L} : Y \nsubseteq L\}\]

\end{proposition}
\begin{proof}
\begin{eqnarray*}
x \in \downarrow{Y} &\Leftrightarrow& \,\exists\,y \in Y, x \triangleleft_{\mathcal{L}} y \\
&\Leftrightarrow& \,\exists\, y \in Y,\,\exists\, L \in \mathcal{L}, x \in L \textrm{ and } y \notin L \\
&\Leftrightarrow& \, \exists\, L \in \mathcal{L}, x \in L, Y \nsubseteq L \\
&\Leftrightarrow& x \in \bigcup \{L \in \mathcal{L} : Y \nsubseteq L\}.
\end{eqnarray*}
\end{proof}

\begin{corollary}\label{corollary - preliminary to cor no maximal}
If $y \in X$ and $\cL$ is a nest on $X$, then:
\[
\downarrow{y} =
 \bigcup \{L \in \mathcal{L} : y \notin L\}.
\]
\end{corollary}

\begin{proposition}\label{corollary - preliminary to cor upper bound}
If $M \in \cL$, then:
\[
\downarrow{M} = \bigcup \{L \in \cL : M \nsubseteq L\} = \bigcup\{L \in \cL : L \subsetneq M\}
\]
and hence $M \in \cL$ is a lower set, iff:
\[M = \bigcup\{L \in \cL : L \subsetneq M\}\]
\end{proposition}
\begin{proof}
By Proposition \ref{proposition - downset} we get that $\downarrow{M} = \bigcup\{L \in \cL : M \nsubseteq L\}$,
which is equal to $\bigcup\{L \in \cL : L \subsetneq M\}$, because $\cL$ is a nest.
\end{proof}

\begin{proposition}\label{corollary - 2.5}
Let $X$ be a set and let $\cL$ be a $T_0$-separating nest, on $X$. Then, $M \in \cL$
is a lower set if and only if $M$ has no maximal element.
\end{proposition}
\begin{proof}
$M$ has no maximal element if and only if for every $x \in M$, there exists $y \in M$, such that $x \triangleL y$.
So, by Corollary \ref{corollary - preliminary to cor no maximal} the proof is complete.
\end{proof}

\begin{proposition}
Let $X$ be a set and let $\cL$ be a nest on $X$. For $M \in \cL$, $X-M \in \cL^c$ is a lower
set in $X$, with respect to the order $\triangleleft_{\cL^c}$, if and only if:
\[M = \bigcap \{L \in \cL : M \subsetneq L\}.\]
\end{proposition}
\begin{proof}
According to Proposition \ref{proposition - downset}, $X-M$ is a lower set, with respect
to $\triangleleft_{\cL^c}$, if and only if
$X-M  = \bigcup \{X-L : L \in \cL, X-M \nsubseteq X-L\}$ if and only if
$M = \bigcap \{L : L \in \cL, X-L \subsetneq X-M\} = \bigcap \{L \in \cL : M \subsetneq L\}$.

\end{proof}

\begin{proposition}
Let $X$ be a set and let $\cL$ be a $T_0$-separating nest on $X$. Then, $X-M \in \cL^c$
is a lower set, with respect to $\triangleleft_{\cL^c}$ if and only if $X-M$ has no
$\triangleL$-minimal element.
\end{proposition}
\begin{proof}
We apply Proposition \ref{corollary - 2.5} for $X-M$. So, $X-M$ is a lower set with
respect to $\triangleleft_{\cL^c}$, if and only if $X-M$ has no maximal element
with respect to $\triangleleft_{\cL^c}$, which is equivalent to the fact that
$X-M$ has no minimal element with respect to $\triangleleft_{\cL}$, since
$\triangleleft_{\cL^c} = \triangleright_{\mathcal{L}}$.
\end{proof}

So, we can now give a characterization of interlocking nests, in terms
of lower sets, without the nest being necessarily $T_0$-separating.

\begin{theorem}\label{proposition-used at the end}
Let $X$ be a set and let $\cL$ be a nest on $X$. Then, $\cL$ is
interlocking, if and only if for each $L \in \cL$, if $X-L$ is
a lower set with respect to $\triangleleft_{\mathcal{L}^c}$, then
$L$ is a lower set with respect to $\triangleL$.
\end{theorem}

{\bf{Open Question 3.}}  Given the characterization of interlocking nest in terms of
lower sets, how can this lead into a restatement of the
orderability problem as (re-)stated in \cite{Good-Papadopoulos}? Will such a restatement
lead into a brand new proof?

\section{Lower and Upper Bounds of Subsets of a set $X$, in Terms of Nests.}

Here we find necessary and sufficient conditions
for the existence of lower and upper bounds, for subsets of a set
$X$, in terms of nests. A general characterization of bounded sets
was given in \cite{lambrinos-boundedness}, as a generalization of
the notion of compactness. We believe that nests can play a
dominant role to the development of this subject.

\begin{theorem}\label{proposition-X equals down Y}
Let $X$ be a set, let $\mathcal{L}$ be a nest on $X$ and let $\triangleleft_{\mathcal{L}}$
be the order induced by $\mathcal{L}$. Let also $Y \subset X$. Then, $X = \downarrow{Y}$, if and only if
there exists a cover of $X$, by elements of $\mathcal{L}$, such that there does not exist
a single-element subcover of $Y$, by this cover.
\end{theorem}
\begin{proof}
We first suppose that there exists a cover $\{L_i:i \in I\} \subset \mathcal{L}$, for $X$,
i.e. $\bigcup_{i \in I} L_i  = X$, such that there does not exist a single-element subcover, for $Y$,
by this cover. Let $x \in X$. Then, there exists $L_j \in \{L_i : i \in I\}$, such that
$x \in L_j$. But $Y \nsubseteq L_j$. So, there exists $y \in Y$, such that $y \notin L_j$.
Thus, $x \in L_j$ and $y \notin L_j$, which gives that $x \triangleleft_{\mathcal{L}} y$,
which finally gives that $x \in \downarrow{Y}$. Consequently, $X \subset \downarrow{Y}$
and thus $X = \downarrow{Y}$.

On the other hand, let us suppose that $X = \downarrow{Y}$. So, for every $x \in X$, there
exists $y \in Y$, such that $x \triangleleft_{\mathcal{L}} y$. So, for every $x \in X$,
there exists $L_x \in \mathcal{L}$, such that $x \in L_x$ and $y \notin L_x$. The latter
implies that $\bigcup_{x \in X} L_x = X$. It now remains to prove that there does not
exist a single-element subcover for $Y$, by this cover. For this, let us suppose that there
exists a single-element subcover for $Y$, by the cover
$\{L_x : x \in X\}$. Then, $Y \subset L_x$, where $x \in X$.
But, $y_x \in Y$ implies that $y_x \in L_x$, which leads into a contradiction.
\end{proof}

\begin{remark}\label{remark - remark for upper bound}
In particular, when $Y \subsetneq X$, by $\downarrow{Y} \subsetneq X$ we mean that
there exists $x \in X$, such that $x \notin \downarrow{Y}$ or, equivalently:
\[\textrm{(there
exists } x \in X, \textrm{ such that for every } y \in Y, \,x \ntriangleleft_{\mathcal{L}} y)~~~~~~~(1)\]
If we suppose that the nest $\mathcal{L}$ $T_0$-separates $X$, then $(1)$ will be equivalent to the statement:
\[\textrm{(there exists } x \in X, \textrm{ such that, for each } y \in Y,\, y \triangleleft_{\mathcal{L}} x)\]
In other words, $Y$ has an upper bound in $X$.
\end{remark}

So, we can now extract the following:

\begin{proposition}\label{corollary-bounded}
Let $\mathcal{L}$ be a $T_0$-separating nest, on $X$. If $Y \subset X$, then $Y$ has an upper bound in $X$, that does
not belong to $Y$, if and only if for each cover of $X$, by elements of $\mathcal{L}$,
there exists a finite subcover for $Y$, by members of this cover.
\end{proposition}

\begin{proposition}
Let $\cL$ be a $T_0$-separating nest on $X$. Then, for every $L \in \cL$, such that $L \neq X$,
$L$ has always an upper bound in $X$, that does not belong to $X$.
\end{proposition}
\begin{proof}
By Proposition \ref{corollary - preliminary to cor upper bound},
$\downarrow{L} = \bigcup \{M \in \cL : M \subset L\} \subsetneq X$. So,
by Remark \ref{remark - remark for upper bound}, $L$ has an upper bound in $X$.
\end{proof}

\begin{remark}
The characterization of Proposition \ref{corollary-bounded} is identical
to the characterization of bounded subsets of a given set
in topological spaces (see \cite{lambrinos-boundedness}). More specifically, if $(X,\mathcal{T})$ is a topological
space and if $A \subset X$, then $A$ is {\em bounded} in $X$, if and only if
for each open cover of $X$, there exists a finite subcover for $A$,
by members of this cover. So, we can say that the subsets of $X$, which have an upper bound on $X$,
are exactly the bounded subsets of $X$.
\end{remark}

\begin{proposition}
Let $X$ be a set, let $\mathcal{L}$ be a nest on $X$ and let $\triangleL$ be
the order induced by $\cL$. Let also $Y \subset X$. Then, $X = \uparrow{Y}$,
if and only if there exists a nest $\cL' = \{L_i : i \in I\} \subset \cL$, with the
property that $\bigcap_{i \in I} L_i = \emptyset$, such that there does not
exist an $L_i \in \cL'$, with the property that
$Y \cap L_i = \emptyset$.
\end{proposition}
\begin{proof}
We first suppose that there exists a nest $\cL' = \{L_i : i \in I\} \subset \cL$
with the property $\bigcap_{i \in I} L_i = \emptyset$, such that there does not
exist an $L_i \in \cL'$, with the property
$Y \cap L_i = \emptyset$. Let $x \in X$. Then, there exists
$i \in I$, such that $x \notin L_i$. So, $Y \cap L_i \neq \emptyset$ implies
that there exists $y \in Y$, such that $y \in L_i$. So, there exists $y \in Y$,
such that $y \in L_i$ and $x \notin L_i$, i.e. $y \triangleL x$. Thus,
$X \subset \uparrow{Y}$, which proves that $X = \uparrow{Y}$.

Conversely, let $X = \uparrow{Y}$. Then, for each $x \in X$, there exists
$y_x \in Y$, such that $y_x \triangleL x$. Thus, for every $x \in X$, there
exists $L_x \in \cL$, such that $y_x \in L_x$ and $x \notin L_x$. Let $\cL' = \{L_x : x \in X\}$.
We observe that $\bigcap_{x \in X}L_x = \emptyset$. If we suppose that there exists
an $L_i \in \cL'$ with the
property $Y \cap L_i = \emptyset$, then $Y \cap L_i = \emptyset$.
Thus, $y_i \in Y$, such that $y_i \in L_i$, a contradiction.
\end{proof}

\begin{remark}
In particular, by $\uparrow{Y} \subsetneq X$, we mean that there exists
$x \in X$, such that $x \notin \uparrow{Y}$ or, equivalently:\\
(there exists $x \in X$, such that for every $y \in Y$, $y \ntriangleleft_{\mathcal{L}} x$)~~~(*) \\
In particular, if $\cL$ is a nest that $T_0$-separates $X$, then (*) will take the form: \\
(there exists $x \in X$, such that for each $y \in Y$, $x \triangleL y$)\\
In other words, $Y$ has a lower bound in $X$.
\end{remark}

\begin{corollary}\label{corollary - prel to cor on upper bound}
Let $X$ be a set and let $\cL$ be a $T_0$-separating nest on $X$. If $Y \subset X$,
then $Y$ has a lower bound in $X$ that does not belong to $Y$, if and only if for each nest $\cL' = \{L_i: i \in I\} \subset \cL$,
such that $\bigcap L_i = \emptyset$, there exists an $L_i \in \cL'$, such that $Y \cap L_i = \emptyset$.
\end{corollary}

\begin{corollary}\label{corollary - for next chapter}
Let $X$ be a set and let $\cL$ be a $T_0$-separating nest, on $X$. Then,
each $L \in \cL$ has no lower bound in $X$, that does not belong to $X$.
\end{corollary}
\begin{proof}
The proof follows immediately from Proposition \ref{corollary - preliminary to cor upper bound}
and the fact that $\cL$ is a nest.
\end{proof}

\section{Nests, Groups and Topological Groups.}

We consider the order $\triangleL$, on a group $(G,*)$, which is
generated by a $T_0$-separating nest of sets, in $G$, and we give
conditions which will make the order compatible with the group operation,
$*$.

Let $(G,*)$ be a group, with operation $*$, and let $\mathcal{L}$ be a $T_0$-separating
nest, on $G$. For every $x,y \in G$, $x \triangleleft_{\mathcal{L}} y$, if and only
if there exists $L \in \mathcal{L}$, such that $x \in L$ and $y \notin L$. The order $\triangleleft_{\mathcal{L}}$
is said to be {\em compatible} with the group operation $*$, if and only if for every $a,b$ and $g$, in $G$, the following hold:
\[a \triangleleft_{\mathcal{L}} b \Leftrightarrow\]\[ a * g \triangleleft_{\mathcal{L}} b * g\] and
\[g * a \triangleleft_{\mathcal{L}} g * b.\]

\begin{proposition}\label{proposition-order compatible with *}
Let $(G,*)$ be a group and let $\mathcal{L}$ be a $T_0$-separating
nest on $G$.
If for every $g \in G$, for every $L \in \mathcal{L}$:
\[g * L \in \mathcal{L}\] and
\[L * g \in \mathcal{L};\]
equivalently, if the maps:
\[t: \mathcal{L} \times G \to \mathcal{L}, \textrm{ where } t(L,g)= L * g\] and
\[s: G \times \mathcal{L} \to \mathcal{L}, \textrm{ where } s(g,L) = g * L\]
are well-defined, then $\triangleleft_{\mathcal{L}}$ is compatible with $*$.
\end{proposition}
\begin{proof}
Let $e \in G$ denote the identity element of $G$, with respect to $*$.
Let, for every $g \in G$ and for every $L \in \mathcal{L}$, $g * L \in \mathcal{L}$ and
$L * g \in \mathcal{L}$. Let $a,b \in G$, such that $a \triangleleft_{\mathcal{L}} b$,
and let also $g \in G$. We prove that $a * g \triangleleft_{\mathcal{L}} b * g$. But,
since $a \triangleleft_{\mathcal{L}} b$, there exists $L \in \mathcal{L}$, such that
$a \in L$ and $b \notin L$. Furthermore, $a \in L$ implies that $a * g \in L * g$ and
$b \notin L$ implies that $b * g \notin L * g$, because if $b * g$ belonged
to $L * g$, then $(b * g) * g^{-1} \in (L * g) * g^{-1}$, which would imply
that $b * e \in L$, which would then imply that $b \in L$, a contradiction. Finally, $a * g \triangleleft_{\mathcal{L}} b * g$.
In a similar way we prove that $g * a \triangleleft_{\mathcal{L}} g * b$.
\end{proof}

\begin{example}
Let $(\mathbb{R},+)$ be the group of the real numbers, under addition. Then, $\mathcal{L} = \{(-\infty,a): a \in \mathbb{R}\}$
is obviously a $T_0$-separating nest on $\mathbb{R}$, and we observe that for every
$b \in \mathbb{R}$, $b+(-\infty,a) = (-\infty,a+b) \in \mathcal{L}$. So, $\triangleleft_{\mathcal{L}}$
is compatible, with respect to $+$.
\end{example}

\begin{example}
Consider the abelian group $(\mathbb{R}-\{0\},\times)$, of the non-zero
real numbers, endowed with the operation of multiplication. Obviously, $\mathcal{L} = \{(-\infty,a): a \in \mathbb{R}\}$ is
a $T_0$-separating nest, on $\mathbb{R}$. We remark that if $b \in \mathbb{R}$,
such that $b \triangleright_{\mathcal{L}} 0$, then $(-\infty,a) \times b = (-\infty,a \times b) \in \mathcal{L}$,
but if $b \triangleleft_{\mathcal{L}} 0$, then $(-\infty,a) \times b = (a \times b,\infty) \notin \mathcal{L}$. So,
$\triangleleft_{\mathcal{L}}$ is not compatible, with respect to $\times$.
\end{example}

{{\bf Open Question 4.}} Since the above examples refer to the connection between
properties of nests and abelian groups in particular, it might be interesting to investigate
examples of non-abelian groups, and see what
topological properties, if any, does $\triangleL$ bring to the structure of a non-abelian group.

We will now make the problem a bit more difficult.

\begin{proposition}
Let $(G,*)$ be a group. Let also $\mathcal{L}$ and $\mathcal{R}$
be families of subsets of $G$. Suppose that the following two conditions are satisfied:
\begin{enumerate}

\item For every $L \in \mathcal{L}$, $L^{-1} \in \mathcal{R}$.

\item For every $R \in \mathcal{R}$, $R^{-1} \in \mathcal{L}$.

\end{enumerate}

If we consider the topology generated by $\mathcal{L} \cup \mathcal{R}$,
then the map $f : G \to G$, where $f(x) = x^{-1}$, will be continuous.
\end{proposition}
\begin{proof}
Let $L \in \mathcal{L}$. Then:
\begin{eqnarray*}
f^{-1}(L) &&= \{x \in G : f(x) \in L\}\\
&&=\{x \in G : x^{-1} \in L\}\\
&&=L^{-1} \in \mathcal{R}.
\end{eqnarray*}

Similarly, if $R \in \mathcal{R}$, then $f(R) = R^{-1} \in \mathcal{L}$.
\end{proof}

\begin{proposition}\label{proposition-top group}
Let $(G,*)$ be a group. Let also $\mathcal{L}$ and $\mathcal{R}$
be families of subsets of $G$. Suppose that the following two conditions are satisfied:
\begin{enumerate}

\item If $x * y \in L \in \mathcal{L}$, then there exist $L_x,L_y \in \mathcal{L}$,
such that $x \in L_x, y \in L_y$ and $L_x * L_y \subset L$.

\item If $x * y \in R \in \mathcal{R}$, then there exist $R_x,R_y \in \mathcal{R}$,
such that $x \in R_x, y \in R_y$ and $R_x * R_y \subset R$.

\end{enumerate}

If we consider the topology generated by $\mathcal{L} \cup \mathcal{R}$,then
the map $f : G \times G \to G$, where $f(x,y) = x * y$, will be continuous.
\end{proposition}
\begin{proof}
Let $L \in \mathcal{L}$. Then, $f^{-1}(L) = \{(x,y) \in G \times G : x * y \in L\}$.
Statement 1. gives that for every $(x,y) \in G \times G$, such that $x * y \in L$,
there exist $L_x,L_y \in \mathcal{L}$, such that $x \in L_x$, $y \in L_y$ and $L_x * L_y \subset L$,
which implies that:
\[L_x \times L_y \subset f^{-1}(L).~~~~~(1)\]
Indeed:
\begin{eqnarray*}
(a,b) \in L_x \times L_y  && \Rightarrow \\
a \in L_x, b \in L_y  && \Rightarrow \\
a * b \in L_x * L_y && \Rightarrow \\
a * b \in L.
\end{eqnarray*}

It is also true that:
\[\pi_1^{-1}(L_x) \cap \pi_2^{-1}(L_y) \subset L_x \times L_y~~~~~(2),\]
where $\pi_1^{-1}(L_x)$ and $\pi_2^{-1}(L_y)$ are the inverse projections,
which give the usual product topology, in $G \times G$.

Indeed:
\begin{eqnarray*}
(a,b) \in \pi_1^{-1}(L_x) \cap \pi_2^{-1}(L_y) && \Rightarrow \\
a \in L_x, b \in L_y && \Rightarrow \\
(a,b) \in L_x \times L_y
\end{eqnarray*}

So, $(1)$ and $(2)$ give that $\pi_1^{-1}(L_x) \cap \pi_2^{-1}(L_y) \subset f^{-1}(L)$.
The latter implies that:
\[\bigcup_{x * y \in L} [\pi_1^{-1}(L_x) \cap \pi_2^{-1}(L_y)] \subset f^{-1}(L)~~~~~(3).\]

But, it also holds that:
\[f^{-1}(L) \subset \bigcup_{x * y \in L} \pi_1^{-1}(L_x) \cap \pi_2^{-1}(L_y)~~~~~(4).\]

Indeed:
\begin{eqnarray*}
(a,b) \in f^{-1}(L) && \Rightarrow \\
f(a,b) \in L && \Rightarrow \\
a * b \in L && \Rightarrow \\
\exists \, L_a,L_b \in \mathcal{L},\,a \in L_a,b \in L_b,\,L_a * L_b \subset L && \Rightarrow \\
(a,b) \in \pi_1^{-1}(L_a) \cap \pi_2^{-1}(L_b) \subset \bigcup_{x * y \in L} [\pi_1^{-1}(L_x) \cap \pi_2^{-1}(L_y)]
\end{eqnarray*}

So, $(3)$ and $(4)$ finally give that:
\[f^{-1}(L) = \bigcup_{x * y \in L} [\pi_1^{-1}(L_x) \cap \pi_2^{-1}(L_y)].\]
and we conclude that $f^{-1}(L)$ is open in $G \times G$. In a similar way,
$f^{-1}(R)$ is open in $G \times G$, too.
\end{proof}

{{\bf Open Question 5.}} Proposition \ref{proposition-top group} refers to any family
of subsets of a set $X$. Will it be possible to prove it by restricting it only to properties of nests? This will hopefully give a characterization of topological groups,
with the involvement of order-theoretic and topological properties nests.

{{\bf Open Question 6.}} We will finally summarize a list of problems worth looking at, concerning
ordered groups. Can nests play any significant role in order to give interesting answers to them?

1.  Suppose $(X,\tau)$ be a topological space. If (and only if) $\tau$
 satisfies a condition $P$ then there is a group operation $*$ on $X$ such that
 $(X,*,\tau)$ is a topological group.

2. Suppose $(X,\tau)$ be a topological space. If (and only if) $\tau$
 satisfies $P$ then there is a group operation $*$, on $X$, and an order
 $<$, on $X$, such that $(X,*,<,\tau)$ is an ordered topological group
 with $<$ inducing $\tau$.

3. Suppose $X$ be a set, $*$ a group operation on $X$ and $\tau$ a
 topology (coming from an order?) on $X$. Are there ``easy'' conditions to see
 that $*$ is continuous with respect to $\tau$?


\bigskip

Email address:\\
kyriakos.papadopoulos1981@gmail.com (Kyriakos Papadopoulos)\\

\end{document}